\documentclass[11pt]{amsart}

\usepackage{amsmath, amsthm, amssymb}
\usepackage{graphicx, array}
\usepackage{longtable}
\usepackage{epstopdf}
\usepackage{mathrsfs}
 \usepackage{setspace}

\theoremstyle{definition}

\newtheorem{theorem}{Theorem}[section]
\newtheorem{proposition}{Proposition}
\newtheorem{lemma}{Lemma} 
\newtheorem{corollary}{Corollary}

\newtheorem{remark}{Remark}

\title{The word problem for a family of one relation Adian inverse semigroups}

\author{Muhammad Inam}
\email{minam@saumag.edu}

\address{Department of Mathematics and Computer Science,\\
Southern Arkansas University\\
Magnolia, AR 71753 USA}

\begin{document}
\date{\today}

\begin{abstract}  The word problem for an Adian inverse semigroup given by the presentation $Inv\langle a,b|a=ba^nb\rangle$, where $n\geq 1$, is decidable.

\end{abstract}

\maketitle

\section{Introduction}

\bigskip

A semigroup $S$ is called an \textit{inverse semigroup} if for every element $a$ of $S$ there exists a unique element $b$ in $S$ such that $b=bab$ and $a=aba$. This unique element $b$ of $S$ is called the \textit{inverse} of $a$, and is denoted by $a^{-1}$. In an inverse semigroup idempotents commute and product of two idempotents is an idempotent. The \textit{natural partial order} on an inverse semigroup $S$ is defined as, for $a,b\in S$,  $a\leq b$ if and only if $a=aa^{-1}b$. A congruence relation $\sigma$ is defined on $S$ as, for $a,b\in S$, $a\sigma b$ if and only if there exists an element $c\in S$ such that $c\leq a,b$. The congruence $\sigma$ turns out to be the minimum group congruence on $S$. So, $S/\sigma$ is the maximal group homomorphic image of $S$. An inverse semigroup $S$ is \textit{$E$-unitary}, if the natural homomorphism from $S$ onto $G=S/\sigma$ is \textit{idempotent pure}, i.e., the inverse image of the identity of $G$ is the set of all idempotents of $S$.   All these facts along with other fundament details of the theory of inverse semigroups can be found in the text \cite{IS}.  

In this paper, $X$ denotes the alphabet, and $R=\{(u_i,v_i)| i\in I\}$, where $u_i,v_i\in X^+$  denotes the set of positive relations. If $(u_i,v_i)\in R$, the words $u_i$ and $v_i$ are called \textit{$R$-words}. The pair $\langle X|R\rangle$ is called is called a \textit{positive presentation}. The semigroup generated by the set $X$ and subject to the set of relations $R$ is denoted by $Sg\langle X|R\rangle$. The group generated by $X$ and subject to the set of relations $R$ is denoted by $Gp\langle X|R\rangle$. Similarly, the inverse semigroup generated by $X$ and admitting $R$ as its set of relations is denoted by $Inv\langle X|R\rangle$.

We can associate two undirected graphs to a positive presentation $\langle X|R\rangle$. The \textit{left graph} is denoted by $LG\langle X|R\rangle$. In $LG\langle X|R\rangle$, the vertices are labeled by the elements of $X$,  there is an edge corresponding to every relation $(u_i,v_i)\in R$, that connects the prefix letter of $u_i$ with the prefix letter of $v_i$. Similarly, the \textit{right graph} of the presentation $\langle X|R\rangle$ is denoted by $RG\langle X|R\rangle$, and it can be constructed by joining the suffix letter of $u_i$ with the suffix letter of $v_i$, for every $(u_i,v_i)\in R$. A closed path in $LG\langle X|R\rangle$ is called a \textit{left cycle}, and a closed path in $RG\langle X|R\rangle$ is called a \textit{right cycle}. We refer to \cite{Thesis} for further details and examples of the left and right graph graphs of a positive presentation. If there is no cycle in the left and right graphs of a positive presentation $\langle X|R\rangle$, then the presentation $\langle X|R\rangle$ is called a \textit{cycle free} presentation. A cycle free presentation is also called an Adain presentation, because these presentations were first studied by S. I. Adain  \cite{Adian}.  

It has been proved by Meakin \cite{Meakin} that for a given positive finite presentation $\langle X|R\rangle$, it is undecidable whether $Sg\langle X|R\rangle$ is an Adian semigroup, whether $Inv\langle X|R\rangle$ is an Adian inverse semigroup, and whether $Gp\langle X|R\rangle$ is an Adian group.

The following fact proved in  \cite{Eu} is crucial to the study of Adian inverse semigroups. 

\begin{theorem} Adian inverse semigroups are $E$-unitary. 
\end{theorem}

It was shown in early 1930's by Magnus \cite{Magnus} that the word problem for one relator group $Gp\langle X|w=1\rangle$ is decidable. So, if $M=Inv\langle X|u=v\rangle$ happens to be an Adian inverse semigroup, then the membership problem for the set of idempotents on $M$, $E(M)$, is decidable. A word $w$ represents an element of $E(M)$ iff $w/\sigma=1$ in $Gp\langle X|u=v\rangle$ (by $E$-unitary) which can be decided by Magnus. 
 
 Few years back in 2020, Gray \cite{RG} proved a crucial result that in general the word problem for one relator $E$-unitary inverse semigroup is undecidable. 

 The following result was initally proved by Adian \cite{Adian} for only finite Adian presentations. Later Remmers \cite{RM} proved the same result by using semigroup diagrams for any Adian presentation. 

\begin{theorem}\label{embedding} An Adian semigroup embeds in the Adian group with the same presentation. 

\end{theorem}

Theorem \ref{embedding} ensures the decidability of the word problem for all one relation Adian semigroups, as they are embedded in the corresponding one relation Adian groups, who have decidable word problem by Magnus \cite{Magnus}. 

Adian conjectured \cite{Adian}, the word problem is decidable for all Adian semigroups.. Sarkisian \cite{Sar} provided a proof of Adian's conjecture, but later Adian found a gap in her proof. So, this is still an open research problem.  The word problem for Adian inverse semigroup is also open. 

Inam in \cite{Paper 2}, and \cite{Paper 3} studied the word problem for Adian inverse semigroups. The following theorem was proved in \cite{Paper 2}.

\begin{theorem}\label{PMT} Let $M=Inv\langle X|R\rangle$ be a finitely presented Adian inverse semigroup. Then the Sch\"{u}tzenberger graph of $w$, for all $w\in (X\cup X^{-1})^+$, is finite if and only if the Sch\"{u}tzenberger graph of $w'$ is finite, for all $w'\in X^+$. 
\end{theorem}

By using Theorem \ref{PMT}, Inam proved the decidability of the word problem for the Baumslag-Solitar inverse semigroups $Inv\langle a,b|ab^m=b^na\rangle$. 

The following theorem was proved in \cite{Paper 3}. 

\begin{theorem}\label{MT2}  Let $M=Inv\langle X|R\rangle$ be a finitely presented Adian inverse semigroup.  The Sch\"{u}tzenberger graph of every positive word  is finite if and only if the subgraphs of $S\Gamma(w)$, for all $w\in X^+$, generated by all the $R$-words are finite.

\end{theorem}

Inam used Theorem \ref{MT2} to show the decidability of  the word problem for the following classes of one relation Adian inverse semigroups. 

\begin{theorem}\label{MT1} Let $M=Inv\langle X|u=v\rangle$ be an Adian inverse semigroup, such that no $R$-word is a subword of the other $R$-word, and the relation $(u,v)$ is in one of the following forms:
\begin{enumerate} 

\item No $R$-word overlaps with itself or with the other $R$-word.

\item One of the $R$-words overlaps with itself, and the other $R$-word neither overlaps with itself nor with the former $R$-word. 

\item Both $R$-words overlap with themselves, there is no overlap between both the $R$-words, and at least one of the $R$-words is not of the form $x^n$, for some $x\in X^+$ and $n\geq 2$.  

\item A prefix of one $R$-word is a suffix of the other $R$-word, no suffix of the former $R$-word is a prefix of  the latter $R$-word, and no $R$-word overlaps with itself. 
\end{enumerate}

Then the word problem is decidable for $M$. \end{theorem}

In this paper we consider the family of one relation Adian inverse semigroups given by the presentation $Inv\langle a,b|a=ba^nb\rangle$ where $n\geq 1$, and prove Theorem \ref{MT3} given below. In this family of one relation Adian inverse semigroups, the $R$-word $a$ is a proper subword of the other $R$-word $ab^nb$, that makes the Sch\"{u}tzenber graph of every word to be an infinite graph  over the presentation $\langle a,b|a=ba^nb\rangle$, that contains an $R$-word as its subword.

\begin{theorem}\label{MT3}
The word problem for one relation Adian inverse semigroups given by the presentation $Inv\langle a,b|a=ba^nb\rangle$, where $n\geq 1$, is decidable. 
\end{theorem}

\section{Preliminaries}

We denote an alphabet by $X$. The set of all positive words on $X$ is denoted by $X^+$. The set $X^+$  forms a semigroup under the binary operation of concatenation of words, and called the \textit{free semigroup} on $X$. The notation $X^*$ denotes the set $X^+\cup\{\epsilon\}$, where $\epsilon$ denotes the empty word. The empty word serves as the identity element under the binary operation of concatenation of word. The set $X^*$ is called the \textit{free inverse monoid} on $X$.  For a congruence $\rho$ generated by a set of relations $R$, $X^+/\rho$  is the \textit{semigroup} presented as $Sg\langle X|R\rangle$, and $X^*/\rho$ is the \textit{monoid} presented as $Mon\langle X|R\rangle$. We write $w_1\equiv w_2$ to express that $w_1$ and $w_2$ are identical words in $X^*$, and write $w_1=w_2$ to express that $w_1$ and $w_2$ are equal in the monoid $Mon\langle X|R\rangle := X^*/\rho$, where $\rho$ is the congruence generated by $R$. If $\rho$ happens to be the Vagner congruence on the monoid $(X\cup X^{-1})^*$, then $FIM(X):=(X\cup X^{-1})^*/\rho$ is the \textit{free inverse monoid} on $X$. For a subset $R$ of  $ (X\cup X^{-1})^*\times (X\cup X^{-1})^*$, if $\tau$ denotes the congruence relation generated by $R\cup \rho$, then $M=Inv\langle X|R \rangle := (X\cup X^{-1})^*/\tau$ is the inverse monoid  presented by the set of generators $X$ and the set of relations $R$. 

A directed graph in which the edges are labeled by elements of $X$ is called a \textit{labeled directed graph} over a set $X$. The $3$-tuple  $(u,x,v)$ to denotes the edge labeled by $x$ with initial vertex $u$ and terminal vertex $v$.  A \textit{path} or \textit{segment} of length $n$ is a sequence of edges 

$$\{(v_0,x_1,v_1), (v_1, x_2, v_2),..., (v_{n-1}, x_n,v_n)\}$$ 

such that the initial vertex of an edge (except the first) equals the terminal vertex of the previous edge. if $v_0=v_n$, the path is a cycle. We say that the path is labeled by the word $w=x_1 x_2....x_n$ if the word $w$ can be read in the graph starting at $v_0$.

An \textit{inverse word graph} over $X$ is a labeled directed graph over $X\cup X^{-1}$ such that the labeling is consistent with an involution, that is, $(u,x, v)$ is an edge from a vertex $u$ to a vertex $v$ if and only if $(v,x^{-1},u)$ is an edge from $v$ to $u$. A \textit{birooted inverse word graph} is an inverse word graph $\Gamma$ with vertices $\alpha,\beta \in V(\Gamma)$ identified as the start and end vertices, respectively. The \textit{language} $L[A]$ of a birooted inverse word graph $A=(\alpha, \Gamma, \beta)$ is the set of words that label a path from $\alpha$ to $\beta$ in $\Gamma$.  In a birooted inverse word graph over a presentation $\langle X|R\rangle$, for each relation $(r,s)\in R$ and two distinct vertices $v_1$ and $v_2$, if $r$ and $s$ can be read along two directed paths going from $v_1$ to $v_2$, then there is a \textit{region} with boundary given by the pair of paths labeled by $r$ and $s$ starting from $v_1$ and ending at $v_2$. Every region is simply connected, and so is homeomorphic to the open disk.In a birooted inverse word graph, for each relation $(r,s)\in R$ and a vertex $v$, if we can find a segment labeled by one side of the relation $r$ starting from the vertex $v$, but we do not find a segment labeled by the other side $s$ of the relation starting from $v$, then the segment labeled by $r$ is called an \textit{unsaturated segment}. 

 J. B. Stephen \cite{SG} introduced the notion of \textit{Sch\"{u}tzenberger graphs} as a tool to study the word problem for inverse semigroups. If $M=Inv\langle X|R\rangle$ is an inverse semigroup then we may consider the corresponding Cayley graph $\Gamma(M,X)$. The vertices of this graph are labeled by the elements of $M$ and there exists a directed edge labeled by $x\in X\cup X^{-1}$ from the vertex labeled by $m_1$ to the vertex labeled by $m_2$ if $m_2 = m_1x$.   The Cayley graph $\Gamma(M,X)$ is not necessarily strongly connected unless $M$ happens to be a group because it may happen that when there is an edge labeled by $x$ from $m_1$ to $m_2$, there is no edge labeled by $x^{-1}$ from $m_2$ to $m_1$ (so, $m_2 = m_1x$, but $m_1 \neq m_2x^{-1}$). The strongly connected components of $\Gamma(M,X)$ are called the \textit{Sch\"{u}tzenberger graphs} of $M$. For any word $u\in (X\cup X^{-1})^*$, the strongly connected component of $\Gamma(M,X)$ that contains the vertex corresponding to  $u$ is the \textit{Sch\"{u}tzenberger graph of $u$} and it is denoted by $S\Gamma(M,X,u)$. If there is no ambiguity about $M$ and $X$, then the Sch\"{u}tzenberger graph of $u$ is simply denoted by $S\Gamma(u)$.  In \cite{SG} it is shown that the vertices of $S\Gamma(M,X,u)$ are precisely those vertices that are labeled by the elements of the $\mathscr{R}$-class of $u$, i.e., $R_u = \{m\in M \mid mm^{-1} = uu^{-1}\}$.

 For any word $u\in (X\cup X^{-1})^*$, it is useful to consider the \textit{Sch\"{u}tzenberger automaton} $(uu^{-1},S\Gamma(M,X,u),u)$ with initial vertex $uu^{-1}\in M$, terminal vertex $u\in M$ and with the Sch\"{u}tzenberger graph of $u$ as the underlying graph. The language accepted by this automaton is a subset of $(X\cup X^{-1})^*$ and will be denoted as $L(u)$.
 \[ L(u) = \{w\in (X\cup X^{-1})^* \mid w \mbox{ labels a path  from } uu^{-1} \mbox{ to } u \mbox{ in } S\Gamma(M,X,u)\}. \]
Here, $u$ and $w$ may be regarded both as elements of $(X\cup X^{-1})^*$ and as elements of $M$. Thus, $L(u)$ may be regarded as a subset of $(X\cup X^{-1})^*$ or as a subset of $M$.

The following result of Stephen \cite{SG} plays a key role in solving the word problem for inverse semigroups.
\begin{theorem}\label{Stephen's thm}
Let $M=Inv\langle X|R\rangle$ and let $u,v\in (X\cup X^{-1})^*$. 
\begin{enumerate}
\item $L(u)= \{w \mid w \geq u$ in the natural partial order on $M\}$.
\item The following are equivalent:
\begin{enumerate}
\item[(i)] $u = v$ in $M$.
\item[(ii)] $L(u)=L(v)$.
\item[(iii)] $u\in L(v)$ and $v\in L(u)$.
\item[(iv)] $(uu^{-1},S\Gamma(M,X,u), u)$ and $(vv^{-1},S\Gamma(M,X,v),v)$ are isomorphic as automata.
\end{enumerate}
\end{enumerate}

\end{theorem}

 We briefly describe the iterative procedure described by Stephen \cite{SG} for building a Sch\"{u}tzenberger graph. Let $Inv\langle X|R\rangle$  be a presentation of an inverse monoid.

   Given a word $u=a_1a_2...a_n\in (X\cup X^{-1})^*$, the \textit{linear graph} of $u$  is the birooted inverse word graph $(\alpha_u,\Gamma_u,\beta_u)$ consisting of the set of vertices

   \begin{center}

   $V((\alpha_u,\Gamma_u,\beta_u))=\{\alpha_u,\beta_u,\gamma_1,...,\gamma_{n-1}\}$

   \end{center}

   and edges

   \begin{center}

   $(\alpha _u,a_1, \gamma _1),(\gamma _1,a_2,\gamma _2),..., (\gamma _{n-2},a_{n-1},\gamma _{n-1}),(\gamma _{n-1},a_n,\beta _u)$,

   \end{center}

    together with the corresponding inverse edges.

    Let $(\alpha , \Gamma ,\beta )$ be a birooted inverse word graph over $X\cup X^{-1}$. The following operations may be used to obtain a new birooted inverse word graph $(\alpha ',\Gamma ',\beta ')$:

    $\bullet$ \textbf{Determination} or \textbf{folding:} Let $(\alpha,\Gamma,\beta)$ be a birooted inverse word graph with vertices $v,v_1,v_2$, with $v_1\neq v_2$, and edges $(v,x,v_1)$ and $(v,x,v_2)$ for some $x\in X\cup X^{-1}$. Then we obtain a new birooted inverse word graph $(\alpha',\Gamma',\beta')$ via taking the quotient of $(\alpha,\Gamma,\beta)$ by the equivalence relation that identifies the vertices $v_1$ and $v_2$ and the two edges. In other words, edges with the same label coming out of a vertex are folded together to become one edge.

    $\bullet$ \textbf{Elementary $\mathscr{P}$-expansion:} Let $r=s$ be a relation in $R$ and $r$ can be read from $v_1$ to $v_2$ in $\Gamma$, but $s$ cannot be read from $v_1$ to $v_2$ in $\Gamma$. Then we define $(\alpha',\Gamma',\beta')$ to be the quotient of $\Gamma \cup (\alpha _s,\Gamma_s,\beta_s)$ by the equivalence relation that identifies vertices $v_1$ and $\alpha_s$ and vertices $v_2$ and $\beta_s$. In other words, we  ``sew" on a linear graph for $s$ from $v_1$ to $v_2$ to complete the other half of the relation $r=s$.

    An inverse word graph is \textit{deterministic} if no folding can be performed and is \textit{closed} if it is deterministic and no elementary expansion can be performed over a presentation $\langle X|R\rangle$. Note that given a finite inverse word graph it is always possible to produce a determinized form of the graph because determination reduces the number of vertices. So, the process of determination must stop after finitely many steps. We also observe that the process of folding is confluent \cite{SG} .

    If $(\alpha_1,\Gamma_1, \beta_1)$ is obtained from $(\alpha,\Gamma,\beta)$ by an elementary $\mathscr{P}$-expansion, and $(\alpha_2,\Gamma_2,\beta_2)$ is the determinized  form of $(\alpha_1,\Gamma_1,\beta_1)$, then we write $(\alpha,\Gamma,\beta)$\\$\Rightarrow (\alpha_2,\Gamma_2,\beta_2)$ and say that $(\alpha_2,\Gamma_2,\beta_2)$ is obtained from $(\alpha, \Gamma,\beta)$ by a  \textit{$\mathscr{P}$-expansion}. The reflexive and transitive closure of $\Rightarrow$ is denoted by $\Rightarrow ^*$.

    For $u\in (X\cup X^{-1})^*$, an \textit{approximate graph} of $(uu^{-1}, S\Gamma(u), u)$ is a birooted inverse word graph $A=(\alpha,\Gamma,\beta)$ such that $u\in L[A]$ and $y\geq u$ holds in $M$ for all $y\in L[A]$. Stephen showed in \cite{SG} that the linear automaton of $u$ is an approximate graph of $(uu^{-1}, S\Gamma(u), u)$. He also proved the following:

    \begin{theorem}\label{closure}
    Let $u\in (X\cup X^{-1})^*$ and let $(\alpha,\Gamma,\beta)$ be an approximate graph of $(uu^{-1},S\Gamma(u), u)$. If $(\alpha,\Gamma,\beta)\Rightarrow^*(\alpha',\Gamma',\beta')$ and $(\alpha',\Gamma',\beta')$ is closed, then $(\alpha',\Gamma',\beta')$ is the Sch\"{u}tzenberger automaton for $u$.
    \end{theorem}

    In \cite{SG}, Stephen showed that the class of all birooted inverse words graphs over $X\cup X^{-1}$ is a co-complete category  and that the directed system of all finite $\mathscr{P}$-expansions of a linear graph of $u$ has a direct limit. Since the directed system includes all possible $\mathscr{P}$-expansions, this limit must be closed. Therefore, by \ref{closure}, the Sch\"{u}tzenberger graph of $u$ is the direct limit.

An elementary $\mathscr{P}$-expansion is said to be \textit{relative to} a birooted inverse word graph $(\alpha,\Gamma,\beta)$ if it can be applied to it. 

    \textbf{Full $\mathscr{P}$- expansion (a generalization of the concept of $\mathscr{P}$-\\ expansion):} A full $\mathscr{P}$-expansion of a finite birooted inverse word graph $(\alpha,\Gamma,\beta)$ over a finite(infinite) presentation $\langle X|R\rangle$  is obtained in the following way:

  $\bullet$ Apply all possible elementary $\mathscr{P}$-expansions on $(\alpha,\Gamma,\beta)$, that are only relative to  $(\alpha,\Gamma,\beta)$, to obtain a new graph denoted by $(\alpha',\Gamma',\beta')$,.

 $\bullet$ Perform all possible foldings in $(\alpha',\Gamma',\beta')$ to obtain its determinized form $(\alpha_1,\Gamma_1,\beta_1)$,
    

    The birooted inverse word graph $(\alpha_1,\Gamma_1,\beta_1)$ is called the full $\mathscr{P}$-expansion of $(\alpha,\Gamma,\beta)$. 

For any word $w\in (X\cup X^{-1})^*$, $(\alpha_0.\Gamma_0(w),\beta_0)$ denotes the linear graph of $w$. We apply Full $\mathscr{P}$-expansion to $(\alpha_i.\Gamma_i(w),\beta_i)$ to obtain $(\alpha_{i+1}.\Gamma_{i+1}(w),\beta_{i+1})$, for $i\in\mathbb{N}\cup\{0\}$.  This recursive process generates a sequence of birooted approximate graphs $\{(\alpha_i,\Gamma_i(w),\beta_i):i\in \mathbb{N}\cup\{0\}\}$. This sequence of birooted approximate graphs converges to the Sch\"{u}tzenberger graph of $w$.  In a finitely presented inverse semigroup $Inv\langle X|R\rangle$, there exists a graph morphism $\phi_i:(\alpha_i,\Gamma_i(w),\beta_i)\to (\alpha_{i+1},\Gamma_{i+1}(w),\beta_{i+1})$, for any $w\in (X\cup X^{-1})^*$, and $i\in \mathbb{N}\cup\{0\}$. If  $Inv\langle X|R\rangle$ happens to be a finitely presented Adian inverse semigroup, and $w\in X^+$, then by Proposition 3 of \cite{Paper 2} the graph morphism $\phi_i:(\alpha_i,\Gamma_i(w),\beta_i)\to (\alpha_{i+1},\Gamma_{i+1}(w),\beta_{i+1})$ turns out to be an embedding, for all $i\in \mathbb{N}\cup\{0\}$. This embedding is as an induced subgraph embedding. Those regions which appear in $(\alpha_i,\Gamma_i(w),\beta_i)$ as a consequence of application of full $\mathscr{P}$- expansion on $ (\alpha_{i-1},\Gamma_{i-1}(w),\beta_{i-1})$, are called the \textit{$i$-th generation regions}, for all $i\in \mathbb{N}$.

\section{Characterization of edges labeled by $b$}

From now on, our alphabet $X$ denotes the set $\{a,b\}$, and $R$ denotes the set $\{(a,ba^nb)\}$, where $n\geq 1$. For a word $w\in (X\cup X^{-1})^*$, the set of all edges that start from a vertex $\gamma$ of an approximate graph of $S\Gamma(w)$ is called the \textit{out-star set} of $\gamma$,  $Star^O(\gamma)$. Similarly, the set of all edges that terminate at a vertex $\gamma$ of an approximate graph of $S\Gamma(w)$ is called the \textit{in-star set} of $\gamma$, $Star^I(\gamma)$. We call a vertex $\gamma$ to be a \textit{special vertex} of an approximate graph of $S\Gamma(w)$ if $Star^O(\gamma)=X$ or $Star^I(\gamma)=X$.  Here $a$ is also an $R$-word. So, in an approximate graph of $S\Gamma(w)$, if $Star^O(\gamma)=X$  and the edge $a$ starting from $\gamma$ is unsaturated or  $Star^I(\gamma)=X$ and the edge $a$ terminating at $\gamma$ is unsaturated, then we call $\gamma$ to be an \textit{unsaturated special vertex}.

\begin{proposition}\label{P1} For all $w\in (X\cup X^{-1})^*$, in $S\Gamma(w)$ over the presentation $\langle X|a=ba^nb\rangle$, where $n\geq 1$, an edge labeled by $b$ can be a boundary edge of at most two regions.

\end{proposition}
 \begin{proof} The letter $b$ is the prefix and suffix of the $R$-word $ba^nb$, The boundary of each region of $S\Gamma(w)$, for some $w\in (X\cup X^{-1})^*$, over the presentation $\langle X|a=ba^nb\rangle$, contains precisely two edges labeled by $b$, one is the initial edge and the other one is the terminal edge  of the segment $ba^nb$, that labels one side of the boundary of the region. 
 
 We assume that an edge labeled by $b$ of $S\Gamma(w)$, for some $w\in (X\cup X^{-1})^*$, is common among three regions labeled by $A,B$, and $C$. If this edge labeled by $b$ happens to be the initial edge of one side, labeled by $ba^nb$, of all these three regions, then these regions will get identified with each other and become one region of $S\Gamma(w)$, as $S\Gamma(w)$ is determinized. Similarly, by the same argument the edge labeled by $b$ cannot be the terminal edge of one side, labeled by $ba^nb$, of all these three regions either.
 
  So,  we are forced to assume that this edge labeled by $b$ is the initial edge of one side of one of the regions $A-$say, and terminal edge of the one side of another region $B-$say. If this edge labeled by $b$ is the initial edge of one side of $C$, then $A$ and $C$ will fold together in $S\Gamma(w)$, as $S\Gamma(w)$ is determinized. If this edge labeled by $b$ is the terminal edge of one side of $C$, then $B$ and $C$ will fold together in $S\Gamma(w)$. 
 
 \end{proof} 

  For $w\in (X\cup X^{-1})^*$, the determinized form of the linear graph of $w$, $(\alpha'_0,\Gamma'_0(w),\beta'_0)$, denoted by $(\alpha_0,\Gamma_0(w),\beta_0)$ constitutes the Munn tree of $w$, $MT(w)$ \cite{IS}. This birooted inverse word graph  with the notation $(\alpha_0,\Gamma_0(w),\beta_0)$ refers to the first iterative step of the Stephen's process of  full $\mathscr{P}$-expansion of constructing the Sch\"{u}tzenber graph of $w$.  

A \textit{pivotal transversal} is a maximal segment of $MT(w)$, for some $w\in (X\cup X^{-1})^*$, that consists of edges labeled by $a$ only. So, it is labeled by $a^k$, for some $k>0$.  

To make the construction of Sch\"{u}tzenberger graph of $w$, for some $w\in (X\cup X^{-1})^*$, whose Munn tree contains only one pivotal transversal easy and effective, we classify the edges labeled by $b$ of $MT(w)$ into three classes, the edges of color $0$, edges of color $1$, and edges of color $2$. Let us also point out that if for some $w\in (X\cup X^{-1})^*$, $MT(w)$ contains more than one pivotal transversal, then this classification of edges labeled by $b$ of $MT(w)$ is not effective.

The impetus of this classification is: an edge of color $0$ will not be a boundary edge of any region of $S\Gamma(w)$, an edge of color $1$ will be the boundary edge of precisely one region of $S\Gamma(w)$, and an edges of color $2$ will be the common boundary edge of two regions of $S\Gamma(w)$, , for all $w\in (X\cup X^{-1})^*$.

For some $w\in (X\cup X^{-1})^*$, such that $MT(w)$ contains only one pivotal transversal, an edge labeled by $b$ of $MT(w)$ is an \textit{edge of color $0$} if it lies on a maximal segment labeled by $b^r$, for some $r>0$, that satisfy the following conditions.

\begin{enumerate}
\item  Neither the initial nor the terminal vertex of the segment $b^r$ is a special vertex. 

\item  The initial vertex of the segment $b^r$ is not followed by a directed segment labeled by  $a^{n+!}$, and no segment labeled by $a^{n+1}$ starts from the  terminal vertex of the segment $b^r$ in $MT(w)$.  

\item The initial vertex of the segment $b^r$ is not followed by a directed segment labeled by  $ba^{n}$, and no segment labeled by $a^{n}b$ starts from the  terminal vertex of the segment $b^r$ in $MT(w)$.  

\end{enumerate}

The above three conditions are not the only conditions that are sufficient for an edge to be the edges of color $0$. There is one more condition that also gives rise to the edges of color $0$ and it is mentioned in the following definition.

For some $w\in (X\cup X^{-1})^*$, such that $MT(w)$ contains only one pivotal transversal, an edge labeled by $b$ of $MT(w)$ is an \textit{edge of color $1$} if it lies on a maximal segment labeled by $b^r$, for some $r>0$, that satisfy the following conditions.

\begin{enumerate}
\item The initial vertex or the terminal vertex  of the segment $b^r$ is a special vertex.

\item There exists a segment labeled by $a^{n+1}$ that terminates at the initial vertex of the segment $b^r$, or there exists a segment labeled by $a^{n+1}$ that starts from the terminal vertex of the segment $b^r$ in $MT(w)$.  

\item There exists a segment labeled by $b^sa^n$, for some $s\in\mathbb{N}$, that terminates at the initial vertex of the segment $b^r$ or the segment $b^r$ terminates at the initial vertex of the segment $a^nb^s$, for some $s\in\mathbb{N}$,.

\begin{enumerate}
\item if $n=1$, and $r\leq s$, then every edge of the segment $b^r$ is of color $1$, otherwise, only the first(last) $s$ number of edges of $b^r$ are the edges of color $1$, while the remaining  edges of the segment $b^r$ are of color $0$. 

\item If $n>1$, the only the initial(terminal) edge of the segment $b^r$ is of color $1$, while the other edges are of color $0$.  

\end{enumerate}
\end{enumerate}

\begin{remark} There is an exception to the above definitions.  A maximal segment $b^r$ of $MT(w)$ that contains only one pivotal transversal, for some $w\in (X\cup X^{-1})^*$, can satisfy condition $(2)$ and $(3)$ simultaneously (that is, there can be segments labeled $a^{n+1}$ and $b^sa^n$ that terminate at the initial vertex of the segment $b^r$, or there can be segments labeled by $a^{n+1}$ and $a^nb^s$ that start from the terminal vertex of the segment $b^r$, in $MT(w)$). In this case every edge of the segment $b^r$ will be an edge of color $1$.

\end{remark}

Figure $1$ displays the pivotal transversal, the edges of color $0$ and color $1$ in $MT(w_1)$, for some $w_1\in (X\cup X^{-1})^*$, with respect to the presentation $\langle X|a=ba^3b\rangle$. 

 \begin{figure}[h!]
\centering
\includegraphics[trim = 0mm 0mm 0mm 0mm, clip,width=2.8in]{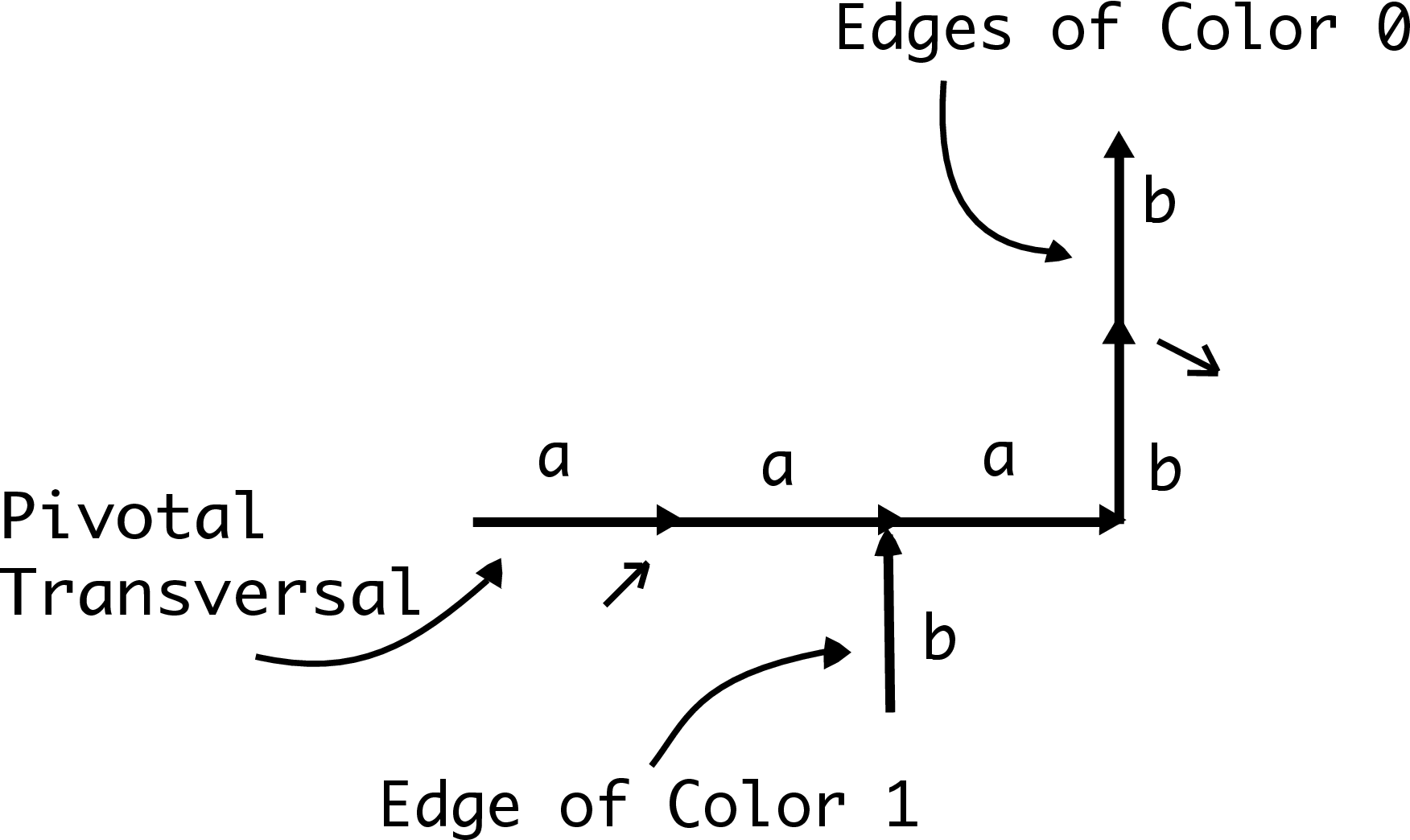}
\caption{$MT(w_1)$, for some $w_1\in (X\cup X^{-1})^*$}
\label{5b}
\end{figure}

For some $w\in (X\cup X^{-1})^*$, in $MT(w)$ that contains only one pivotal transversal, once all the edges of color $0$ and color $1$ recognized, then all the remaining edges labeled by $b$ of $MT(w)$ are the \textit{edges of color $2$}. These edges lie on those maximal segments of the form $b^r$ that satisfy multiple conditions, of the above definition of edges of color $1$, simultaneously.

Figure $2$ displays the pivotal transversal, the edges of color $1$ and color $2$ in $MT(w_2)$, for some $w_2\in (X\cup X^{-1})^*$, with respect to the presentation $\langle X|a=ba^3b\rangle$. 

 \begin{figure}[h!]
\centering
\includegraphics[trim = 0mm 0mm 0mm 0mm, clip,width=3.8in]{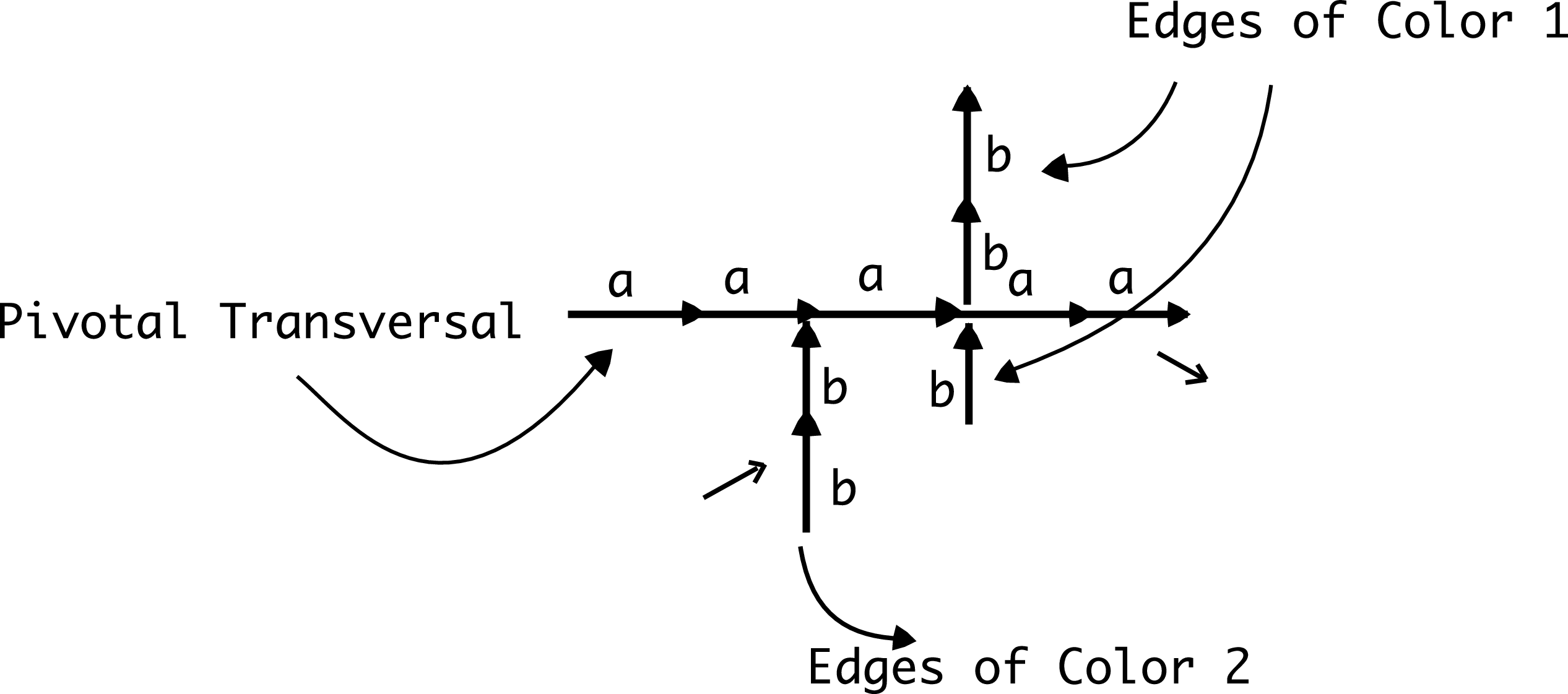}
\caption{$MT(w_2)$, for some $w_2\in (X\cup X^{-1})^*$}
\label{5b}
\end{figure}

 The characterization of edges of color $0$ provided above is established through Proposition \ref{P4}.

\begin{proposition}\label{P4} An edge of color $0$ of $MT(w)$ that contains only one pivotal transversal, for some $w\in (X\cup X^{-1})^*$, is not a boundary edge of any region of $S\Gamma(w)$.
\end{proposition}

\begin{proof}  The proof consists of two parts. The first part deals with those maximal segments of the form $b^r$ of $MT(w)$ that satisfy all the three conditions given in the definition of edges of color $0$. The second part deals with those maximal segments of the form $b^r$   of $MT(w)$ that satisfy exclusively the third condition of the definition of edges of color $1$. Since $MT(w)$ contains only one pivotal transversal, the successive applications of full $\mathscr{P}$-expansion on $(\alpha_0,\Gamma_0(w),\beta_0)$ will not cause any folding of edges that can change the color of an edge labeled by $b$ of $MT(w)$ determined at the beginning. 

\textit{Part I:} We assume that $MT(w)$,  for some $w\in (X\cup X^{-1})^*$ such that $MT(w)$ contains only one pivotal transversal,  contains a maximal segment labeled by $b^r$, for some $r>0$, that satisfy all the three conditions given in the definition of edges of color $0$. We show that the initial edge and all other edges that lie after the initial edge on the segment $b^r$ cannot be the boundary edge of any region of $S\Gamma(w)$. A dual argument can be used to show that the terminal edge and all other edges that lie before the terminal edge on the  segment $b^r$ cannot be the boundary edge of any region of $S\Gamma(w)$. 

The initial edge of $b^r$  cannot get identified with a boundary edge of a region of $S\Gamma(w)$. By condition $(1)$ of edges of color $0$, there is no edge labeled by $a$ that starts from the initial vertex of the segment labeled by $b^r$, therefore, we cannot sew a segment labeled by the other side of the relation $ba^nb$ and fold the initial edge of $b^r$ with the initial edge of the segment $ba^nb$. 
 
 The initial edge of $b^r$ cannot be the terminal edge of a segment labeled by the $R$-word $ba^nb$. By condition $(3)$ of edges of color $0$, there is no segment labeled by $ba^n$  in $MT(w)$ that terminates at the initial vertex of the segment $b^r$. The condition $(2)$ of edges of color $0$ prevents the occurrence of a segment labeled by $ba^n$ that can terminate at the initial vertex of the segment $b^r$ even after one application of full $\mathscr{P}$-expansion on $(\alpha_0,\Gamma_0(w),\beta_0)$.

 Hence the initial edge of the segment $b^r$ cannot be a boundary edge of any region of $S\Gamma(w)$. The second edge of the segment $b^r$ (if it exists) cannot be the boundary edge of any region as well. The segment $b^r$ does not contain any $R$-word, The initial edge of $b^r$ is not a boundary edge of any region, therefore, there cannot be an edge labeled by $a$ that either starts or terminates at the initial vertex of the second edge of the segment $b^r$. This means that the subsegment labeled by $b^{r-1}$ starting from the initial vertex of second edge of the segment $b^r$ satisfy all the three conditions of edges of color $0$. So, by the same argument used above the second edge of the segment cannot be a boundary edge of any region of $S\Gamma(w)$. Continuing in this manner, it follows that none of the edges of the segment $b^r$ can be the boundary edge of any region of $S\Gamma(w)$.

\textit{Part II:} We assume that $MT(w)$, for some $w\in (X\cup X^{-1})^*$, contains a maximal segment labeled by $b^r$, for some $r>0$, exclusively satisfies the third condition of the definition of edges of color $1$. That is, there exists a segment labeled by $b^sa^n$, with $s>0$, that terminates at the initial vertex of the segment $b^r$ in $MT(w)$.  A dual argument can be used for the case where there exists a segment labeled by $a^nb^s$, with $s>0$ starting from the terminal vertex of the segment $b^r$ in $MT(w)$.

There are two possibilities to discuss. 

$i$. If $n>1$, we find a segment labeled by the $R$-word $ba^nb$ starting from the last edge of the segment $b^s$, going along the segment $a^n$ and terminating at the initial edge of the segment $b^r$ in $(\alpha_0, \Gamma_0(w), \beta_0)$. After one application of elementary $\mathscr{P}$-expansion on this unsaturated segment, we cannot find any other unsaturated segment labeled by an $R$-word that utilizes  the second edge of the segment $b^r$, regardless of the fact $s\geq 1$ The second edge (if it exists) of the segment $b^r$ cannot fold to a boundary edge of any other region, as the segment $b^r$ exclusively satisfies the third condition of the definition of edges of color $1$. So, all the  other edges (if they exist) of the segment $b^r$ will not be the boundary edge of any region of $S\Gamma(w)$.

$ii$. If $n=1$ and $s<r$, then it is easy to see that we can find an unsaturated segment labeled by the $R$-word $bab$ starting from the last edge of the segment $b^s$, going along the segment $a$ and terminating at the initial edge of the segment $b^r$ in $(\alpha_0, \Gamma_0(w), \beta_0)$. We apply elementary $\mathscr{P}$-expansion on this unsaturated segment and if $s$ happens to be greater than $1$, we find another unsaturated segment labeled by the $R$-word $bab$ starting from the second last edge of the segment $b^s$ and terminating at the second edge of the segment $b^r$. Continuing in this manner after $s$ number of successive applications of elementary $\mathscr{P}$-expansions will make the first $s$ number of edges  of the segment $b^r$ to be the boundary edges of a region of $S\Gamma(w)$. In $(\alpha_s, \Gamma_s(w), \beta_s)$, we neither find a segment labeled by $bab$ that terminates at the $(s+1)$-th edge of the segment $b^r$, nor the initial vertex of the $(s+1)$-th edge is an unsaturated special vertex. Hence the $(s+1)$-th edge is not a boundary edge of any region. For the same reason, the remaining $(r-s)$ edges of the segment $b^r$ are the edges of color $0$, and they cannot be the boundary edges of any region.

\end{proof}

 The characterization of edges of color $1$ provided above is established through Proposition \ref{P5}.

\begin{proposition}\label{P5} An edge of color $1$ of $MT(w)$ that contains only one pivotal transversal, for some $w\in (X\cup X^{-1})^*$, is on the boundary of precisely one region of $S\Gamma(w)$.
\end{proposition}

\begin{proof} Let $b^r$ be a maximal segment of $MT(w)$ that satisfy precisely one of the three conditions given in the above definition of edges of color $1$. The key fact towards the proof of this proposition is that, the maximal segment $b^r$ does not satisfy any two conditions simultaneously (except possibly condition $(2)$ and $(3)$ as remarked above), therefore, the edges of segment $b^r$ cannot be the common boundary edges of two regions of $S\Gamma(w)$.  Since $MT(w)$ contains only one pivotal transversal, the successive applications of full $\mathscr{P}$-expansion on $(\alpha_0,\Gamma_0(w),\beta_0)$ will not cause any folding of edges that can change the color of an edge labeled by $b$ of $MT(w)$ determined at the beginning.

It is easy to see that if the  segment $b^r$ satisfies only one of the two cases mentioned in condition $(1)$ of the definition of edges of color $1$, then $r$ number of successive applications of  full $\mathscr{P}$-expansions on  $(\alpha_0,\Gamma_0(w),\beta_0)$ will make  every edge of the maximal segment $b^r$ to be a boundary edge of one region of $S\Gamma(w)$.

 if the  segment $b^r$ satisfies exactly one of the two cases mentioned in condition $(2)$ of the definition of edges of color $1$, then $(r+1)$ number of successive applications of  full $\mathscr{P}$-expansions on  $(\alpha_0,\Gamma_0(w),\beta_0)$ will make  every edge of the maximal segment $b^r$ to be a boundary edge of one region of $S\Gamma(w)$.

 iI the  segment $b^r$ satisfies exclusively one of the two cases mentioned in condition $(3)$  of the definition of edges of color $1$, then there are three possibilities to discuss. Without loss of generality, we assume that there exists a segment labeled by $b^sa^n$ that terminates at the initial vertex fo the segment $b^r$.
 
 \begin{enumerate}
 
 \item If $n=1$ and $s< r$, then  $s$ number of successive applications of  full $\mathscr{P}$-expansions on  $(\alpha_0,\Gamma_0(w),\beta_0)$ will make  the first  $s$ number of edges of the maximal segment $b^r$ to be the boundary edges of $s$ district region of $S\Gamma(w)$. The remaining $(r-s)$ edges of the segment $b^r$ will be the edges of color $0$ (see the proof of Proposition \ref{P4}). 
 
 \item If $n=1$ and $s\geq r$, then $r$ number of successive applications of  full $\mathscr{P}$-expansions on  $(\alpha_0,\Gamma_0(w),\beta_0)$ will make  all the edges of the segment $b^r$ to be the boundary edges  $r$ distant regions of $S\Gamma(w)$.

 \item If $n>1$, then the first application full $\mathscr{P}$-expansion on  $(\alpha_0,\Gamma_0(w),\beta_0)$ will make  the first edge of the maximal segment $b^r$ to be a boundary edge  of a first generation region of $S\Gamma(w)$. Since $n>1$, there is no unsaturated segment labeled by $ba^nb$ that terminates at the second edge of the segment $b^r$ in $(\alpha_1,\Gamma_1(w),\beta_1)$. Also, the segment $b^r$ does not satisfy any other condition given in the definition of edges of color $1$. Therefore, the remaining $(r-1)$ edges of the segment $b^r$ cannot be the boundary edges of any region of $S\Gamma(w)$. So, the remaining $(r-1)$ edges of the segment $b^r$ are the edges of color $0$ (see proof of Proposition \ref{P4}).  
 \end{enumerate}  
\end{proof}

We conclude this section at the following Lemma. 

\begin{lemma}\label{L7} 
 If  $MT(w)$, for some $w\in (X\cup X^{-1})^*$, contains only one pivotal transversal, and $p$ denotes the maximum of the lengths of the segments of the form $b^r$,  for some $r\in\mathbb{N}$, that either start from or terminate at the pivotal transversal, then $(\alpha_{p+1},\Gamma_{p+1}(w),\beta_{p+1})$ over the presentation $\langle X|a=ba^nb\rangle$, does not contain any unsaturated special vertex. 
\end{lemma}

\begin{proof} For a given word $w\in (X\cup X^{-1})^*$, such that $MT(w)$contains only one pivotal transversal, let $p\in \mathbb{N}$ be the maximum of the lengths of the segments $b^r$ that either start from or terminate at  a vertex of the pivotal transversal.  We identify the segments of color $0$, color $1$, and color $2$ in $MT(w)$, with respect to the presentation $\langle X|a=ba^nb\rangle$. 

By Proposition \ref{P4}, the edges of color $0$ cannot be the boundary edges of any region of $S\Gamma(w)$. Therefore, the segments consisting of the edges of color $0$ can be ignored.

 We focus our attention towards the segments of color $1$ and $2$. It follows from the definition of segments of color $1$ and $2$ that if $b^r$, for some $r>0$, is a segment of color $1$ or $2$, then every edge of this segment becomes completely saturated (that is, becomes a boundary edge of one or two regions) after at most $(r+1)$ successive applications of full $\mathscr{P}$-expansions starting from $(\alpha_0,\Gamma_0(w),\beta_0)$. Thus, after $(p+1)$ successive applications of full $\mathscr{P}$-expansions, we obtain the approximate graph $(\alpha_{p+1},\Gamma_{p+1}(w),\beta_{p+1})$ in which every edge of color $1$ and color $2$ is completely saturated, and there are no more unsaturated edges labeled by $b$ of any color left that can create any unsaturated special vertex. So, this approximate graph does not contain any unsaturated special vertex.  

\end{proof}

\section{Main theorem}

The following lemmas are needed for the proof of Theorem \ref{T2}. 

\begin{lemma}\label{P6} If  $(\alpha_i,\Gamma_i(w), \beta_i)$, for some $w\in (X\cup X^{-1})^*$ and $i\in \mathbb{N}$, does not contain any unsaturated special vertex, then $(\alpha_i,\Gamma_i(w), \beta_i)$ embeds into $(\alpha_{i+1},\Gamma_{i+1}(w), \beta_{i+1})$ as a bi-rooted inverse word graph, over the presentation $\langle X|a=ba^nb\rangle$.
\end{lemma}

\begin{proof} We assume $(\alpha_i,\Gamma_i(w),$  $ \beta_i)$, for some $w\in (X\cup X^{-1})^*$, and $i\in\mathbb{N}$ does not contain any unsaturated special vertex. This assumption leads to the fact, the initial and terminal vertices of the unsaturated segment labeled by the $R$-words (if there are any) of $(\alpha_i,\Gamma_i(w), \beta_i)$ are not the unsaturated special vertices. So, when we apply elementary $\mathscr{P}$-expansions on these unsaturated segments labeled by the $R$-words, we cannot perform any folding. 

Therefore, the natural graph morphism from $(\alpha_i,\Gamma_i(w), \beta_i)$ to $(\alpha_{i+1},\Gamma_{i+1}$ $(w), \beta_{i+1})$ is an embedding, because  no two distict edges or vertices of $(\alpha_i,\Gamma_i(w), \beta_i)$ get identified with each other in $(\alpha_{i+1},$ $\Gamma_{i+1}(w),$ $ \beta_{i+1})$.

\end{proof}

It can be deduced from Lemma \ref{P6} that if for some $w\in (X\cup X^{-1})$, $i$ happens to be the least number such that $(\alpha_i,\Gamma_i(w), \beta_i)$ does not contain any unsaturated special vertex, then $(\alpha_j,\Gamma_j(w), \beta_j)$ embeds into $(\alpha_{j+1},\Gamma_{j+1}(w), \beta_{j+1})$ for all $j\geq i$. 

\begin{lemma}\label{P7}  If $(\alpha_i,\Gamma_i(w), \beta_i)$, for some $w\in (X\cup X^{-1})^*$ and $i\in \mathbb{N}$, does not contain any unsaturated special vertex, then $(\alpha_{i+1},\Gamma_{i+1}(w), \beta_{i+1})$ does not contain any unsaturated special vertex either, over the presentation $\langle X|a=ba^nb\rangle$.
\end{lemma}

\begin{proof} We assume that there exists a positive integer $n$ such that $(\alpha_n,\Gamma_n(w),$  $ \beta_n)$, for some $w\in (X\cup X^{-1})^*$, does not contain any unsaturated special vertex. By Lemma \ref{P6},  $(\alpha_n,\Gamma_n(w), \beta_n)$ embeds into $(\alpha_{n+1},\Gamma_{n+1}(w), \beta_{n+1})$. 

Let $\delta$ denotes a vertex of $(\alpha_{n+1},\Gamma_{n+1}(w), \beta_{n+1})$.  There are two possibilities to discuss.

\begin{enumerate}
\item If  $\delta\in V\{(\alpha_{n+1},\Gamma_{n+1}(w), \beta_{n+1})\setminus (\alpha_n,\Gamma_n(w), \beta_n)\}$, then $Star^O(\delta)$ and $Star^I(\delta)$ contain only one element. 
So, $\delta$ cannot be a special vertex. 

\item If $\delta$ denotes the image of a vertex of $(\alpha_n,\Gamma_n(w), \beta_n)$ in $(\alpha_{n+1},$ $\Gamma_{n+1}(w),$ $ \beta_{n+1})$, then $\delta$ cannot be an unsaturated special vertex, because of our hypothesis. 

\end{enumerate}

\end{proof}

Lemma \ref{P7} leads us to the conclusion  if for some $w\in (X\cup X^{-1})$,  $i$ happens to be the least positive integer such that $(\alpha_i,\Gamma_i(w), \beta_i)$ does not contain any unsaturated special vertex, then $(\alpha_j,\Gamma_j(w), \beta_j)$ does not contain any unsaturated special vertex for all $j\geq i$.

In an approximate graph $(\alpha_k,\Gamma_k(w),\beta_k)$ of $S\Gamma(w)$, for some $w\in (X\cup X^{-1})^*$, over the presentation $Inv\langle X|a=ba^nb\rangle$, an unsaturated special vertex $\delta$ involve  an unsaturated edge labeled by $a$ and an unsaturated edge labeled by $b$ such that either both of these edges start from $\delta$ or both terminate at $\delta$. 

The edges labeled by $b$ of $S\Gamma(w)$, for some $w\in(X\cup X^{-1})^*$, can be distributed in two classes.

\begin{enumerate}

\item The first class consists of those edges labeled by $b$ of $S\Gamma(w)$ that are the images of $MT(w)$ under the natural graph morphism.

\item While the second class consists of those edges labeled by $b$ of $S\Gamma(w)$, that appear as a consequence of the successive applications of full $\mathscr{P}$-expansions.

\end{enumerate}
The next lemma proves that the edges in the second class mentioned above do not create the unsaturated special vertices in an approximate graph of $S\Gamma(w)$.

\begin{lemma}\label{L4} For some $w\in (X\cup X^{-1})^*$ and $k\in\mathbb{N}$, let $e$ be an edge labeled by $b$ such that $e\in E((\alpha_k,\Gamma_k(w),\beta_k)\setminus (\alpha_{k-1},\Gamma_{k-1}(w),\beta_{k-1}))$, over the presentation $\langle X|a=ba^nb\rangle$. Then the initial and terminal vertices of $e$ are not the unsaturated special vertices of $(\alpha_k,\Gamma_k(w),\beta_k)$.

\end{lemma}

\begin{proof} Since $e\in E((\alpha_k,\Gamma_k(w),\beta_k)\setminus (\alpha_{k-1},\Gamma_{k-1}(w),\beta_{k-1}))$, the edge $e$ appears as a consequence of sewing on a segment labeled by $ba^nb$ from the initial vertex to the terminal vertex of an unsaturated segment labeled by $a$ of $(\alpha_{k-1},\Gamma_{k-1}(w),\beta_{k-1})$. Without loss of generality, we assume that $e$ is the first edge of this newly attached segment, and we show that the initial and the terminal vertices of $e$ cannot be the unsaturated special vertices of $(\alpha_k,\Gamma_k(w),\beta_k)$.

The initial vertex of $e$ cannot be an unsaturated special vertex of  $(\alpha_k,$ $\Gamma_k(w),$ $ \beta_k)$. Because, if the initial vertex of $e$ happens to be an unsaturated special vertex of  $(\alpha_k,\Gamma_k(w),\beta_k)$, then this is only possible when there are two distinct edges labeled by $a$, with same initial vertex in $(\alpha_{k-1},\Gamma_{k-1}(w),\beta_{k-1})$. This leads to the contradiction that $(\alpha_{k-1},\Gamma_{k-1}(w),\beta_{k-1})$ is not determinized.

The terminal vertex of $e$ cannot be an unsaturated special vertex of  $(\alpha_k,\Gamma_k(w),\beta_k)$. Because, the terminal vertex of $e$ can be a special vertex of  $(\alpha_k,\Gamma_k(w),\beta_k)$ in two possible ways.

\begin{enumerate}
\item  If there is a segment labeled by $ba^{-1}$ that starts from the initial vertex of $e$, in $(\alpha_{k-1},\Gamma_{k-1}(w),\beta_{k-1})$,  then the edge $e$ gets identified with the initial edge of the preexisting segment $ba^{-1}$, and creates a special vertex at the terminal vertex of $e$. But, in this case there was already a special vertex on the segment $ba^{-1}$ that did not appear as a consequence of sewing on the segment $ba^nb$. So, when we apply full $\mathscr{P}$-expansion on $(\alpha_{k-1},\Gamma_{k-1}(w),\beta_{k-1})$, the terminal vertex of $e$ becomes a saturated special vertex. 

\item  If there is a segment labeled by $b^{-1}a^{-n-1}$ starting from the terminal vertex of the newly attached segment  $ba^nb$, in $(\alpha_{k-1},\Gamma_{k-1}(w),\beta_{k-1})$, the segments $b^{-1}a^{-n-1}$ starting from the terminal vertex of the newly attached segment $ba^nb$ get identified with the subsegment labeled by $a^nb$ of the newly attached segment and creates a special vertex at the terminal vertex of $e$. when we apply full $\mathscr{P}$-expansion on $(\alpha_{k-1},\Gamma_{k-1}(w),\beta_{k-1})$, every edge labeled by $a$ of the preexisting segment $b^{-1}a^{-n-1}$, becomes saturated. After performing all possible foldings, the terminal vertex of $e$ becomes s saturated special vertex. 

\end{enumerate}

If $e$ happens to be the last edge of the newly attached segment $ba^nb$, then a dual argument to the above can be presented to show that the initial and terminal vertices of $e$ cannot be the unsaturated special vertices of $(\alpha_k,\Gamma_k(w),\beta_k)$.

\end{proof}

It follows from Lemma \ref{L4}, an unsaturated special vertex of an approximate graph $(\alpha_k,\Gamma_k(w),\beta_k)$ of $S\Gamma(w)$, for some $w\in (X\cup X^{-1})^*$,  does not involve an edge labeled by $b$ from the second class. It means that an unsaturated special vertex always involves an edge labeled by $b$ that belongs to the first class.  This means that if $k$ is chosen large enough so that all the edges labeled by $b$ in the first class are completely saturated in $(\alpha_k,\Gamma_k(w),\beta_k)$,  then there will be no unsaturated special vertices in $(\alpha_k,\Gamma_k(w),\beta_k)$. The next Lemma suggests an appropriate choice of $k$ for which $(\alpha_k,\Gamma_k(w),\beta_k)$ does not contain any unsaturated special vertex.

We call a directed segment labeled by $b^r$, for some $r>0$, that connects two pivotal transversals, to be a \textit{connecting segment}, and the terminal vertex of a connecting segment to be a \textit{cut vertex} of $MT(w)$, for some $w\in (X\cup X^{-1})^*$.  

\begin{lemma}\label{L5}  For any $w\in (X\cup X^{-1})^*$, let $m=p+1$, where $p$ is the maximum of the lengths of the directed segments that are labeled by $b^r$,  for some $r\in\mathbb{N}$, oriented in any direction,, of $MT(w)$. Then $(\alpha_m,\Gamma_m(w),\beta_m)$ does not contain any unsaturated special vertex.

\end{lemma}

\textbf{\textit{Idea of the proof:}}  For an arbitrary word $w\in (X\cup X^{-1})^*$, the Stephen's process of full $\mathscr{P}$-expansion does not guarantee the non-existence of unsaturated special vertices in an approximate graph. So, we adopt a slightly different approach for constructing $(\alpha_m,\Gamma_m(w),\beta_m)$ that ensures that $\Gamma_m(w)$ does not contain any unsaturated special vertex. 

We start from $MT(w)$. We distribute $MT(w)$ into a finite number of fragments, such that each fragment contains only one pivotal transversal. We obtain a finite approximate graph from each fragment by applying full $\mathscr{P}$-expansion, successively (as shown in section 3) that does not contain any unsaturated special vertex. We further expand these finite approximate graphs to the ones that contain all of their $m$-generation regions by applying full $\mathscr{P}$-expansion, successively. We connect all these finite approximate graphs at the same vertices of their pivotal transversals from where they were removed, and perform all possible foldings.  We show that this new graph does not contain any unsaturated special vertex. 

It is possible that this new graph may not contain all the regions of $\Gamma_m(w)$. So, to include all the remaining regions up to the $m$-th generation in this graph we apply an iterative process (similar to the full $\mathscr{P}$-expansion) on this graph. We show that no new unsaturated special vertex gets created at any step of this iterative process.

\begin{proof} If $MT(w)$, for some $w\in (X\cup X^{-1})^*$, does not contain any pivotal transversal, then the above statement holds true vacuously.  

If $MT(w)$, for some $w\in (X\cup X^{-1})^*$, contains only one pivotal transversal, then the above statement follows from Lemma \ref{L7}, and Lemma \ref{P7}. 

If $MT(w)$, for some $w\in (X\cup X^{-1})^*$, contains more than one pivotal transversal, then the proof of the above statement can be given as follows. 

For a given word $w\in (X\cup X^{-1})^*$, we construct $MT(w)$. The graph $MT(w)$ can contain only a finite number of pivotal transversals, as $w$ is a word of finite length. We assume that there are $q$ number of pivotal transversals in $MT(w)$, for some $q\in\mathbb{N}$. Therefore, there are $(q-1)$ connecting segments and cut vertices in $MT(w)$. 

 \begin{figure}[h!]
\centering
\includegraphics[trim = 0mm 0mm 0mm 0mm, clip,width=1.7in]{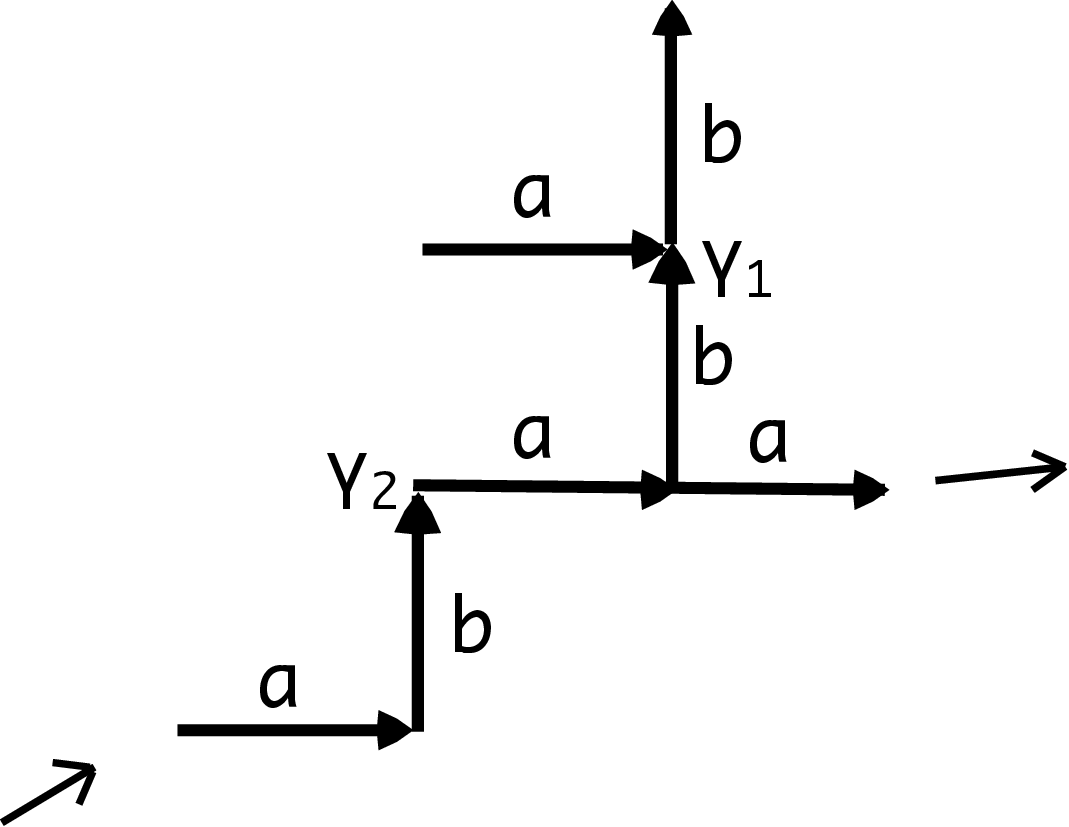}
\caption{$MT(w1)$ with three pivotal transversals and $p=2$, for some $w\in (X\cup X^{-1})^*$}
\label{5b}
\end{figure}

It is possible that some pivotal transversals may not contain any cut vertex, while some other pivotal transversals may contain more than one cut vertex. We label these cut vertices of $MT(w)$ by $\gamma_1,\gamma_2, ...,\gamma_{(q-1)}$ (see figure $3$). We distribute $MT(w)$ into $q$ fragments by cutting $MT(w)$ at its cut vertices. Each of these fragments contain only one pivotal transversal.   We denote these fragments by $\Pi_1, \Pi_2, ...,\Pi_q$ (see figure $4$). 

 \begin{figure}[h!]
\centering
\includegraphics[trim = 0mm 0mm 0mm 0mm, clip,width=1.8in]{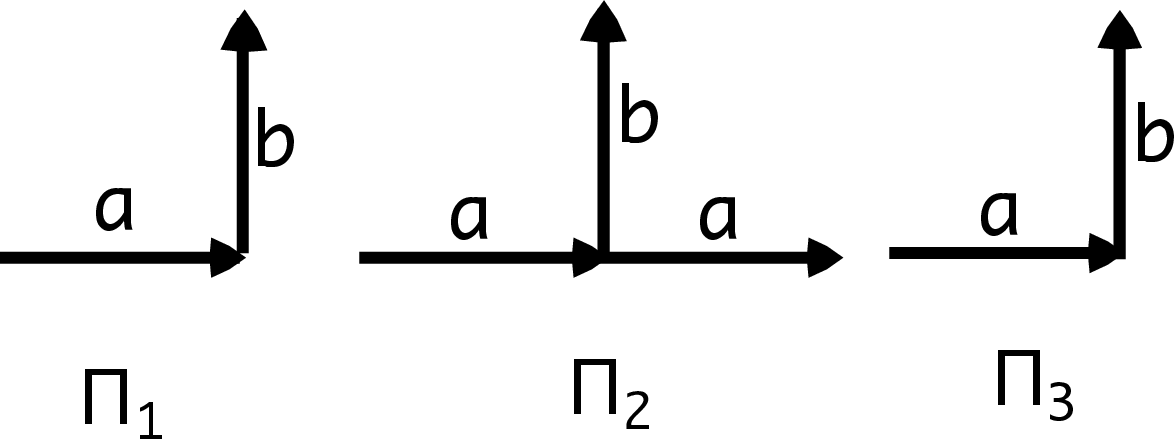}
\caption{$MT(w)$ distributed into three fragments}
\label{5b}
\end{figure}

We assume that $p_1,p_2, ..., p_q$ be the maximums of the lengths of the directed segments of the form $b^r$, oriented in any direction, of the fragments $\Pi_1, \Pi_2, ...,\Pi_q$, respectively. We apply full $\mathscr{P}$-expansions on each $\Pi_i$ successively up to  $(p_i+1)$-times, and denote the corresponding resulting graphs by $\Lambda_1, \Lambda_2, ...,\Lambda_q$, respectively (see figure $5$). 
By Lemma \ref{L7}, $\Lambda_i$'s do not contain any unsaturated special vertex, which means that every edge labeled by $b$ of $\Pi_i$ is completely saturated in $\Lambda_i$ (relative to $\Pi_i$),  for all $1\leq i\leq q$. 

 \begin{figure}[h!]
\centering
\includegraphics[trim = 0mm 0mm 0mm 0mm, clip,width=1.8in]{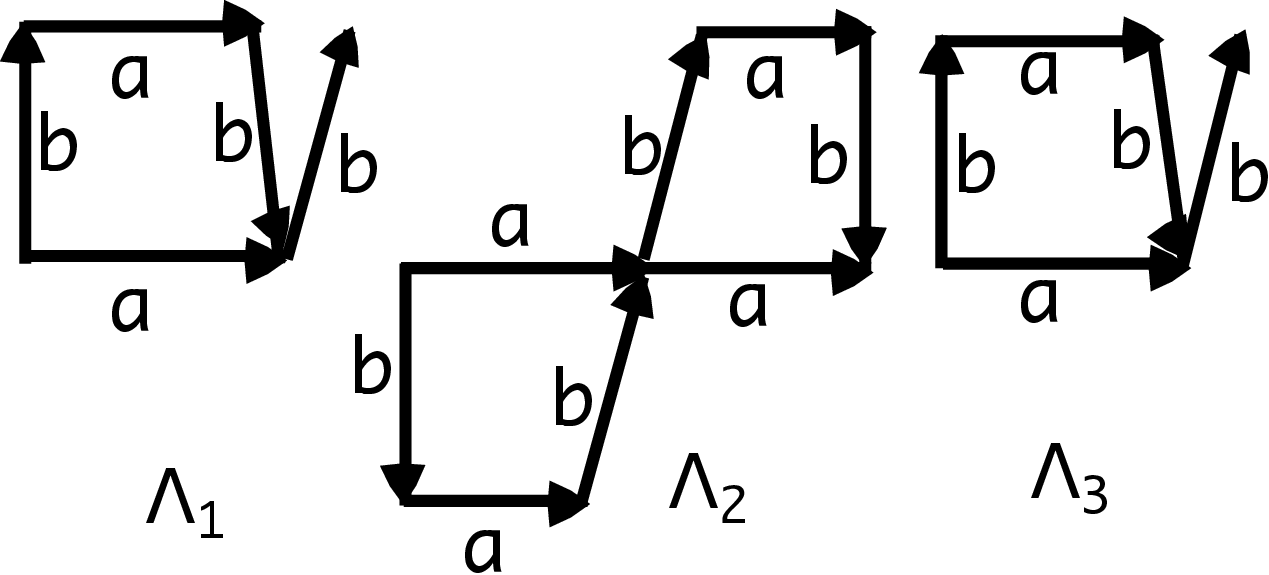}
\caption{$\Lambda_i$ obtained from $\Pi_i$, for $1\leq i\leq 3$, over the presentation $\langle X|a=bab\rangle$}
\label{5b}
\end{figure}

Since $m=p+1>p\geq max\{p_1,p_2, ..., p_q\}$, we obtain $\Delta_i$ from $\Lambda_i$,  by $(m-p_i-1)$ successive applications of  full $\mathscr{P}$-expansions, for $1\leq i\leq q$ (see figure $6$).  

 \begin{figure}[h!]
\centering
\includegraphics[trim = 0mm 0mm 0mm 0mm, clip,width=1.8in]{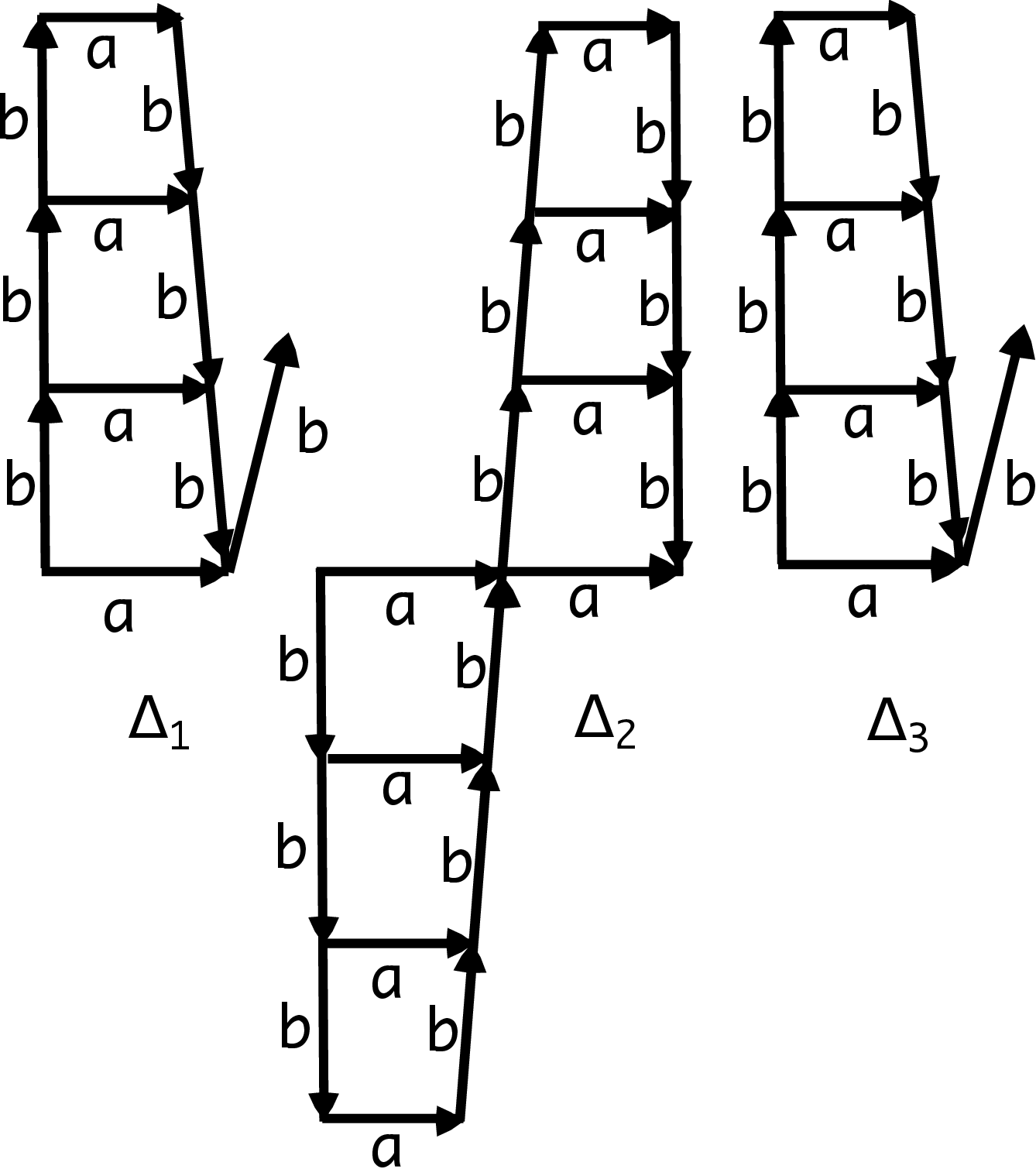}
\caption{$\Delta_i$ obtained from $\Lambda_i$, for $1\leq i\leq 3$, over the presentation $\langle X|a=bab\rangle$}
\label{5b}
\end{figure}

By Lemma \ref{P7}, none of the finite graphs $\Delta_1, \Delta_2, ..., \Delta_q$ contain any unsaturated special vertex relative to the corresponding  $\Pi_i$, where $1\leq i\leq q$. We connect these finite graphs with each other by connecting each $\gamma_i$, for $1\leq i\leq q-1$, at the same vertex of the pivotal transversal of $\Delta_j$, for $1\leq j\leq q$, from where it was removed.  We denote the resulting graph by $\Theta'$. 
 We perform all possible foldings in $\Theta'$ and denote its determinized form by $\Theta$ (see figure $7$). 
 
 \begin{figure}[h!]
\centering
\includegraphics[trim = 0mm 0mm 0mm 0mm, clip,width=1.8in]{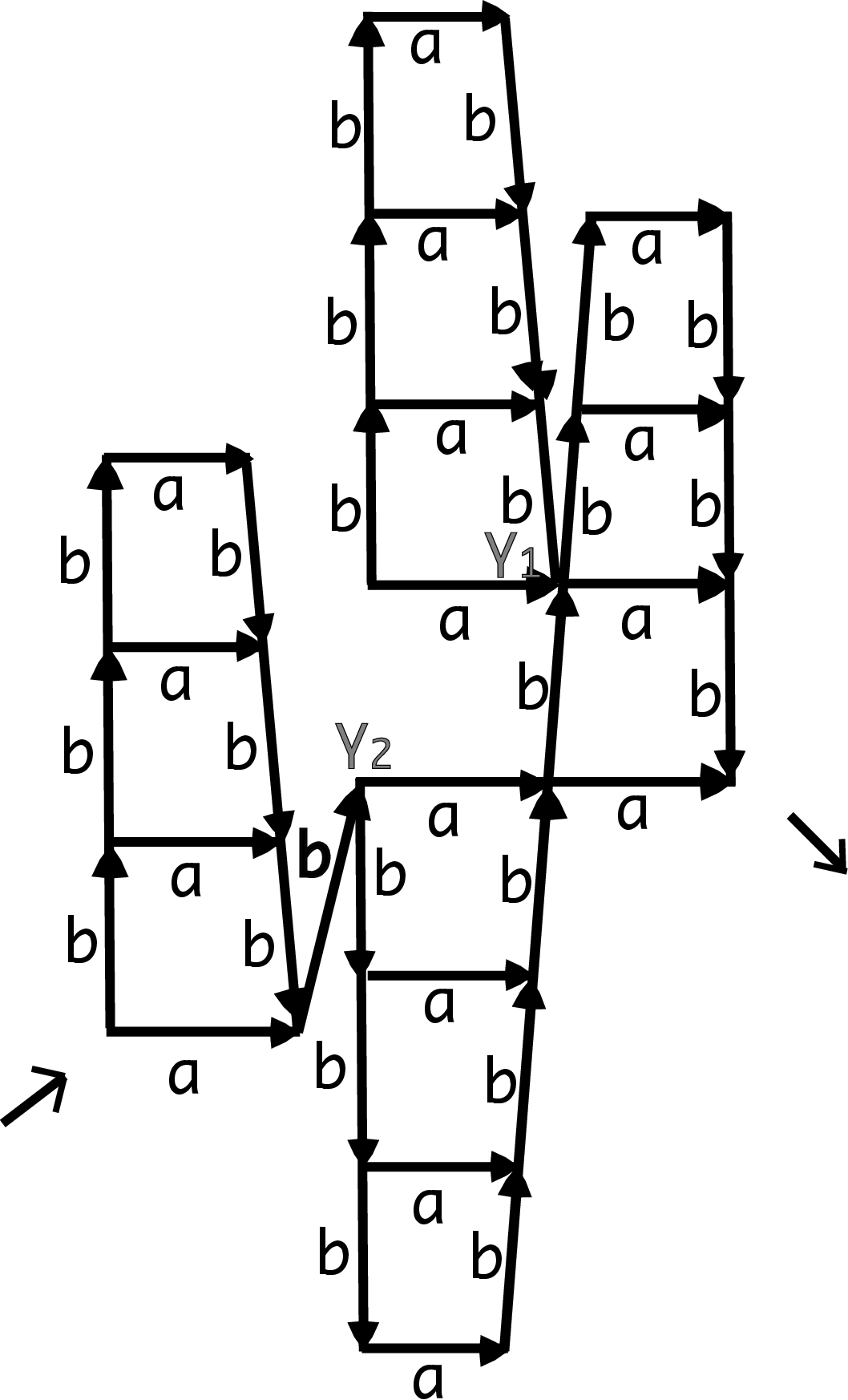}
\caption{$\Theta$ obtained by connecting all $\Delta_j$'s at $\gamma_i$'s and performing all possible foldings}
\label{5b}
\end{figure}

 In $\Theta'$, the out-star set and the in-star set of each cut vertex $\gamma_i$, for $1\leq i\leq q-1$ gets changed. So, this change may give rise to:
 
 \begin{enumerate}
 
 \item some new unsaturated special vertices in $\Theta$ that did not exist in any of the $\Delta_i$'s, 
 
 \item and some new unsaturated segments labeled by the $R$-word $ba^nb$ that pass through some of the $\gamma_i$'s in $\Theta$. 
 
 \end{enumerate}

 \textbf{\textit{Claim:}} The finite graph $\Theta$ does not contain any unsaturated special vertex. 
 
\begin{proof} \textit{Proof of the claim:} We only focus on the connecting segments of $MT(w)$, because these segments contain the cut vertices. 
 
 If neither the initial vertex nor the terminal vertices of a maximal connecting segment between two pivotal transversals of $MT(w)$ are the unsaturated special vertices, then such a segment will not contain any unsaturated special vertex in $\Theta$. Therefore, we can ignore all such connecting segments. 
 
 It is easy to see that if a maximal connecting segment (that is not a proper subsegment of any larger segment of the form $b^r$ in $MT(w)$) is connecting more than two pivotal transversals, then it must pass through an unsaturated special vertex in $MT(w)$. 
 
 We assume that a maximal connecting segment labeled by $b^t$, for some $t\in\mathbb{N}$, contains an unsaturated special vertex  $\delta$, such that $Star^O(\delta)=X$, in $MT(w)$, and $\delta$ is the first such unsaturated special vertex along the orientation of the segment $b^t$.  
 
 The edge labeled by $a$ starting from $\delta$ lies on some $\Pi_i$. The finite graph $\Delta_i$, generated from $\Pi_i$, contains all the first $m$-generation regions of $\Pi_i$. When we connect $\Delta_i$ at $\delta$ in $\Theta'$, the subsegment of $b^t$ from $\delta$ to the terminal vertex of this segment gets identified with a segment of $\Delta_i$. So, all the edges of this subsegment become the boundary edges of some regions of $\Theta$. Hence, in $\Theta$, there is no unsaturated special vertex $\gamma$ left  on the segment $b^t$, such that $Star^O(\gamma)=X$.
 
 Similarly, now we assume that a maximal connecting segment labeled by $b^t$, for some $t\in \mathbb{N}$, contains an unsaturated special vertex $\theta$, such that $Star^I(\theta)=X$, and $\theta$ is the last such unsaturated special vertex along the orientation of the segment $b^t$.
 
 The edge labeled by $a$ ending at $\theta$ lies on some $\Pi_j$. The finite graph $\Delta_j$, generated from $\Pi_j$, contains all the first $m$-generation regions of $\Pi_j$. When we connect $\Delta_j$ at $\theta$ in $\Theta$, the subsegment of $b^t$ from $\theta$ to the initial vertex of this segment gets identified with a segment of $\Delta_j$. So, all the edges of this subsegment become the boundary edges of some regions of $\Theta$. Hence, in $\Theta$, there is no unsaturated special vertex $\gamma$ left  on the segment $b^t$, such that $Star^I(\gamma)=X$.
\end{proof}

It follows from the construction of $\Theta$ that the images of edges labeled by $b$ of $MT(w)$ that could create an unsaturated special vertex in an approximate graph of $S\Gamma(w)$, have been completely utilized in $\Theta$. So, no expansion of $\Theta$ (obtained through an elementary $\mathscr{P}$-expansion and folding) can have an unsaturated special vertex. 
 
 Next, we draw our attention towards  those unsaturated segments labeled by the $R$-word $ba^nb$ that pass through some of the $\gamma_i$'s in $\Theta$. We use the following iterative procedure to obtain a finite determinized graph $\Gamma$ that contains all the first $m$ generation regions of $S\Gamma(w)$, and it does not contain any unsaturated special vertex.

 In $\Theta$, If there are some unsaturated segments labeled by $ba^nb$ that are passing through some of the cut vertices, $\gamma_i$'s, we apply  elementary $\mathscr{P}$-expansions on all such segments, and denote the resulting graph by $\Theta'_1$. We have no reason to believe that $\Theta_1'$ is determinized. So, we perform all possible foldings in $\Theta'_1$ and denote its determinized form by $\Theta_1$. Here, it is worth while to realize that  the $R$-word $a$ consists of only one letter, therefore, it just labels one edge of any approximate graph of $S\Gamma(w)$. So, as a consequence of foldings in $\Theta'_1$ no new edge labeled by $a$ gets created that did not exist before. Hence, if $\Theta_1$ contains an unsaturated segment labeled by an $R$-word, it must be labeled by $ba^nb$. The graph $\Theta_1$, being an expansion of $\Theta$, does not contain any unsaturated special vertex.

 If $\Theta_1$ contains some unsaturated segments labeled by $ba^nb$, we apply elementary $\mathscr{P}$-expansions on all such segments and denote the resulting graph by $\Theta'_2$. We perform all possible foldings in $\Theta'_2$, and denote its determinized form by $\Theta_2$. Just like the previous step, if $\Theta_2$ contains an unsaturated segment labeled by an $R$-word that did not exist in $\Theta_1$, it must be labeled by $ba^nb$. Also, $\Theta_2$, being an expansion of $\Theta$, does not contain any unsaturated special vertex. 
 
  \begin{figure}[h!]
\centering
\includegraphics[trim = 0mm 0mm 0mm 0mm, clip,width=1.8in]{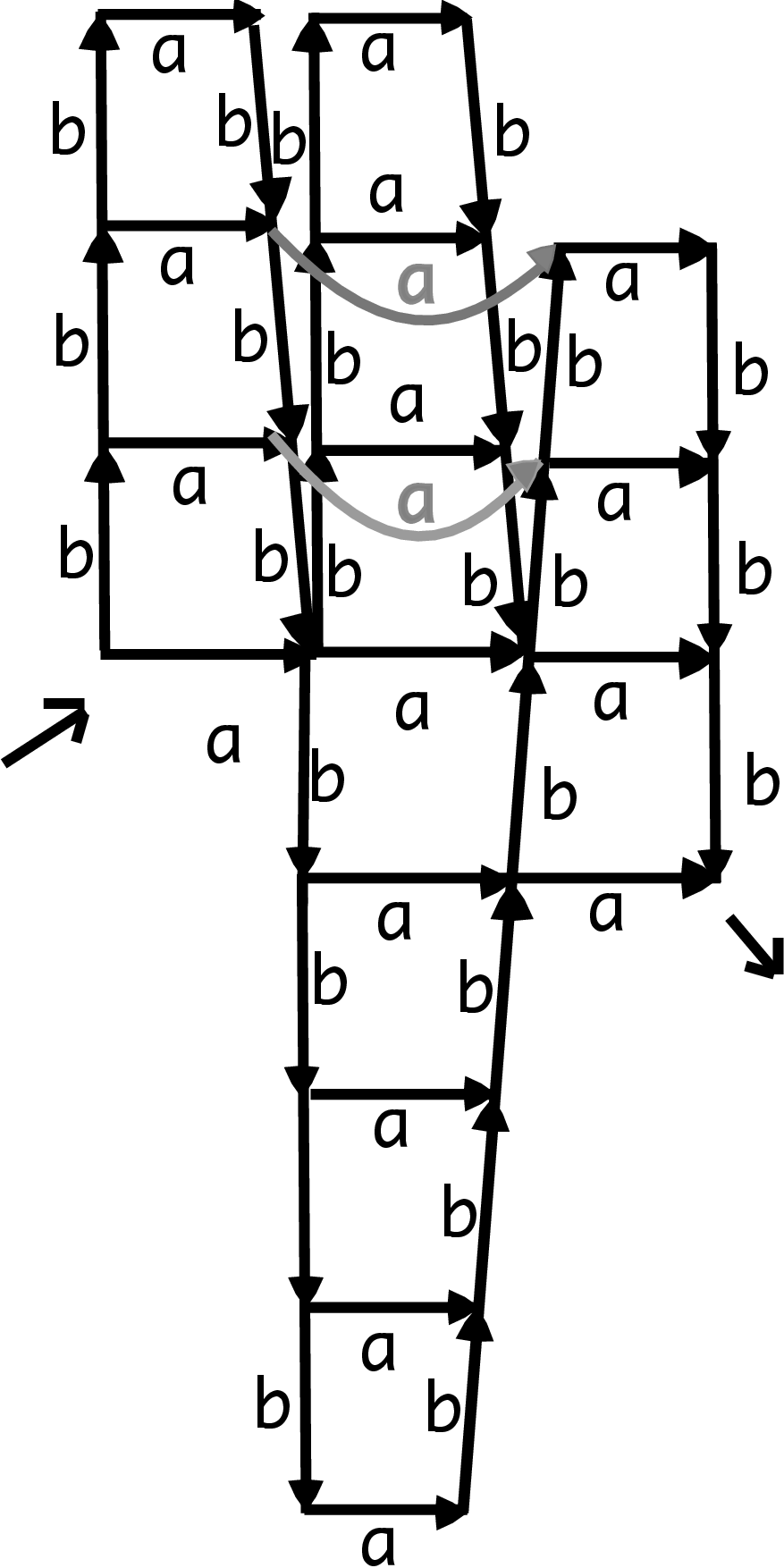}
\caption{$\Gamma$ constrcted from $\Theta$ by using the above iterative process, over the presentation $\langle X|a=bab\rangle$}
\label{5b}
\end{figure}

 We repeat this process at most $m$ times. Finally, we obtain a determinized finite graph that does not contain any unsaturated special vertex, and it contains all the $m$ generation regions of $S\Gamma(w)$. We designate this graph by $\Gamma$ (see figure $8$). Hence, $\Gamma$ is the underlying graph of $(\alpha_m,\Gamma_m(w),\beta_m)$.

\end{proof}

\begin{theorem}\label{T2} For any $w\in (X\cup X^{-1})^*$,  the membership problem for $L(w)$ over the presentation $Inv\langle X|a=ba^nb\rangle$, where $n\geq 1$, is decidable. 
\end{theorem}

\begin{proof} For a given word $w\in (X\cup X^{-1})^*$, let $m=p+1$, where $p$ is the maximum of the lengths of the directed segments of the form $b^r$, for some $r\in \mathbb{N}$, oriented in any direction, of $MT(w)$. By Lemma \ref{L5}, the approximate graph $(\alpha_m,\Gamma_m(w),\beta_m)$, obtained by $m$ successive applications of full $\mathscr{P}$-expansions on $(\alpha_0, \Gamma_0(w),\beta_0)$, of $S\Gamma(w)$, does not contain any unsaturated special vertex. 
By Lemma \ref{P7}, it follows that for $j\geq m$, none of the approximate graph $(\alpha_j,\Gamma_j(w),\beta_j)$ contains any unsaturated special vertex. So, $(\alpha_i,\Gamma_i(w), \beta_i)$ embeds in $(\alpha_{i+1},\Gamma_{i+1}(w_1), \beta_{i+1})$ for $i\geq m$, by Lemma \ref{P6}.
This leads to the conclusion that for any vertex $\gamma\in V( (\alpha_{m+i},\Gamma_{m+i}(w), \beta_{m+i})\setminus (\alpha_{m+i-1},\Gamma_{m+i-1}(w), \beta_{m+i-1}))$, the distance from $\alpha_{m+i}$ to $\gamma$ is greater than or equal to $i$, for all $i\in \mathbb{N}$. Therefore, for any word $w'\in (X\cup X^{-1})^*$ of length$-k-$say, we can consider the finite bi-rooted approximate graph $(\alpha_{m+k},\Gamma_{m+k}(w), \beta_{m+k})$. If $w'$ labels a path from $\alpha_{m+k}$ to $\beta_{m+k}$ in $ \Gamma_{m+k}(w)$, then $w'\in L(w)$, otherwise $w'\notin L(w)$.

\end{proof}

The following Corollary is an immediate consequence of Theorem \ref{T2}.   

\begin{corollary} The word problem is decidable for an Adian inverse semigroup given by the presentation $Inv\langle a,b|a=ba^nb\rangle$, where $n\geq 1$. 

\end{corollary}

\bibliographystyle{plain}

\end{document}